\documentclass[a4paper,10pt,oneside]{amsart}
\usepackage{textalpha} % For unicode greek letters in normal text
\usepackage{amssymb, amsmath, amsthm, mathrsfs, enumitem, amsfonts, latexsym, bbm, stmaryrd, thmtools}
\usepackage{url}
\usepackage{graphicx}
\usepackage{xcolor}
\usepackage{tikz-cd}
\definecolor{svlinks}{rgb}{.0,0.3,0.6} %tmavě modrá
\usepackage[bookmarks,colorlinks=true,pdfhighlight=/O,linkcolor=svlinks,urlcolor=svlinks,citecolor=svlinks,
            pdftitle={Chains of P-points},
            pdfauthor={Dilip Raghavan, Jonathan L. Verner},
            pdfsubject={},
            pdfkeywords={}
            ]{hyperref}
\usepackage{enumitem}
\usepackage[margin=3cm]{geometry}
\usepackage{fouriernc}

\swapnumbers
\declaretheorem[numberwithin=section]{theorem}
\declaretheorem[numbered=no,name=Theorem]{theorem*}
\declaretheorem[sibling=theorem]{proposition}
\declaretheorem[numbered=no,name=Proposition]{proposition*}

\declaretheorem[numbered=no,name=Corollary]{corollary*}
\declaretheorem[sibling=theorem]{fact}
\declaretheorem[numbered=no,name=Fact]{fact*}
\declaretheorem[numbered=no,name=Folklore]{folklore*}
\declaretheorem[sibling=theorem]{observation}
\declaretheorem[numbered=no,name=Observation]{observation*}

\declaretheorem[numbered=no,name=Lemma]{lemma*}
\declaretheorem[sibling=theorem]{question}
\declaretheorem[numbered=no,name=Question]{question*}

\swapnumbers

\declaretheorem[numbered=no,name=Claim]{claim*}
\swapnumbers

\theoremstyle{definition}

\declaretheorem[numbered=no,name=Definition]{definition*}

\declaretheorem[numbered=no,name=Note]{note*}
\declaretheorem[numbered=no,name=Notation]{notation*}

\newcommand{\cc}{\mathfrak{c}}

\newcommand\ZFC{\mathrm{ZFC}}

\newcommand\MA{\mathrm{MA}}

\newcommand\CH{\mathrm{CH}}

\newcommand{\UUU}{{\mathcal{U}}}
\newcommand{\VVV}{{\mathcal{V}}}
\newcommand{\WWW}{{\mathcal{W}}}

\newcommand{\seq}[4]{\langle {#1}_{#2}: #2 #3 #4 \rangle}
\newcommand{\ord}[1]{\##1}
\newcommand{\pre}[2]{#1^{-1}(#2)}

\begin{document}

\title{Chains of P-points}
\keywords{Rudin-Keisler order, ultrafilter, P-point}
\subjclass[2010]{Primary 03E50, 03E05, 54D80.}
\author[Raghavan]{Dilip Raghavan}
\email{dilip.raghavan@protonmail.com}
\address{
Department of Mathematics\\
National University of Singapore\\
Singapore 119076\\
}
\urladdr{http://www.math.nus.edu.sg/~raghavan}
\thanks{The first author was partially supported by National University of Singapore
research grant number R-146-000-211-112.}
\author[Verner]{Jonathan L. Verner}
\email{jonathan.verner@ff.cuni.cz}
\urladdr{http://logic.ff.cuni.cz/people/verner}
\address{
Department of Logic\\
Faculty of Arts\\
Charles University\\
n\'am. Jana Palacha 2\\
116 38 Praha 1\\
}
\thanks{The second author was supported by the joint FWF-GA\v{C}R grant no. 17-33849L:
Filters, ultrafilters and connections with forcing, by the Progres grant Q14. Krize racionality a modern\'\i\ my\v{s}len\'\i\ and by the first author's grant number R-146-000-211-112 from the National University of Singapore}
\begin{abstract}
  It is proved that the Continuum Hypothesis implies that any sequence of rapid P-points of length $< {\cc}^{+}$ which is increasing with respect to the Rudin-Keisler ordering is bounded above by a rapid P-point.
  This is an improvement of a result from \cite{Raghavan:2016}.
  It is also proved that Jensen's diamond principle implies the existence of an unbounded strictly increasing sequence of P-points of length ${\omega}_{1}$ in the Rudin-Keisler ordering.
  This shows that restricting to the class of rapid P-points is essential for the first result.
\end{abstract}
\maketitle

\section{Introduction}

The Rudin-Keisler ordering on ultrafilters, introduced in the late sixties (\cite{Katetov:1968}; see also \cite{Rudin:1966} and \cite{Rudin:1971}),
turned out to be a very useful tool for studying properties of ultrafilters. A variant
of this ordering, the Rudin-Frol{\' i}k ordering, was used by Frol{\' i}k (\cite{Frolik:1967}) to prove,
in $\ZFC$, that the space of non-principal ultrafilters on \(\omega\) is non-homogeneous.
Many combinatorial properties can be characterized in terms of the ordering, e.g.
selective (or Ramsey) ultrafilters are precisely those which are minimal in the Rudin-Keisler
ordering, Q-points are those that are minimal in the Rudin-Blass ordering, P-points
are those below which the Rudin-Keisler and Rudin-Blass orderings coincide.

The first comprehensive study of the Rudin-Keisler (RK) order was done by A.~Blass in his
thesis \cite{Blass:1970}. A.~Blass continued his investigations by considering
the lower part of the ordering viz. the ordering of P-points \cite{Blass:1973}.
He showed that, under suitable assumptions, the ordering can be very rich. Assuming
Martin's Axiom (MA), he showed that

\begin{itemize}
 \item there are \(2^{\mathfrak c}\) many minimal P-points
 \item there are no maximal P-points
 \item the ordering of P-points is \(\sigma\)-closed, both downwards and upwards
 \item the real line as well as \(\omega_1\) can be embedded into the P-points.
\end{itemize}

These results were later extended by several authors (e.g. \cite{Rosen:1985},
\cite{Laflamme:1989}, \cite{Raghavan:2017}). The results that motivated the research that went into
this paper were obtained by B. Kuzeljevi\'c and D. Raghavan~\cite{Raghavan:2016}.
They showed
\begin{theorem*}[Kuzeljevi\'c and Raghavan] Assume $\MA$. The ordinal \(\mathfrak c^+\)
can be embedded into the ordering of (rapid) P-points.
\end{theorem*}
Since any ultrafilter has at most \(\mathfrak c\)-many RK-predecessors, the
above is the best possible result as far as embedding of ordinals is concerned.
The authors of \cite{Raghavan:2016} used the notion of a \(\delta\)-generic
sequence of P-points (see \cite{Raghavan:2016}) which allowed them to carry through
an inductive construction of length \(\mathfrak c^+\).

In this paper we improve upon their results as follows.
In the first part of the
paper (\autoref{thm:closed}) we show that, assuming the Continuum Hypothesis (CH), the ordering of \emph{rapid} P-points is, in fact,
\(\mathfrak c^+\)-closed:
\begin{theorem*} Assume CH. Any increasing sequence of rapid P-points of length
\(<\mathfrak c^+\) is bounded above by a rapid P-point.
\end{theorem*}
Unlike many earlier results, this theorem is more than just an embedding result, for it provides new information about the global structure of the class of rapid P-points under the Rudin-Keisler ordering.

We also show in Section \ref{sec:short} (\autoref{thm:unbounded}) that the fact that we are looking at \emph{rapid} P-points is crucial.
Assuming $\diamondsuit$ (though we suspect that CH is enough), we construct an increasing sequence of P-points of length \(\omega_1\) without any P-point upper bound.

The chains of P-points of length ${\cc}^{+}$ constructed in \cite{Raghavan:2016} enjoy a slightly stronger property than the long chains that can be built using the technique from Section \ref{sec:noshort} of this paper.
The chains of \cite{Raghavan:2016} are all increasing in the ${\leq}^{+}_{\textrm{RB}}$ ordering, but our technique is insufficient to ensure this property for any of the chains of length ${\cc}^{+}$ here.
Thus the existence statement proved in \cite{Raghavan:2016} is stronger than the existence result that is derivable from the work in Section \ref{sec:noshort}.

We should also comment on our assumptions. Since S.~Shelah (\cite{Wimmers:1982})
showed that P-points need not exist at all, or there might be, e.g., just one (see Chapter VI of \cite{Shelah:1998}),
some assumption which guarantees that the structure is rich is needed. For
simplicity we use CH, though a weaker assumption, e.g.\@ MA, would be sufficient for our results.
\section{Preliminaries}
In this section we introduce the basic notions and state some standard facts.

\begin{definition*}[\cite{Rudin:1956}] An ultrafilter \(\mathcal U\) on \(\omega\) is a \emph{P-point}
provided that for any sequence \(\langle X_n:n<\omega\rangle\)
of elements of \(\mathcal U\) there is an \(X\in\mathcal U\) such that
\(|X\setminus X_n|<\omega\) for each \(n\). The last condition will also be
denoted by \(X\subseteq^* X_n\).
\end{definition*}

The following is an alternate characterization which we will often use:

\begin{folklore*} An ultrafilter \(\mathcal U\) is a P-point iff every function
\(f:\omega\to\omega\) is either constant or finite-to-one on some set in
\(\mathcal U\).
\end{folklore*}

\begin{definition*} Given a family \(\mathcal P\) of functions from \(\omega\) to
\(\omega\) we say that a function \(f:\omega\to\omega\) \emph{dominates} \(\mathcal P\)
if \(g\leq^*f\) for each \(g\in\mathcal P\), where
\[
  g\leq^* f\iff (\forall^\infty n)(g(n)\leq f(n)),
\]
and where \(\forall^\infty n\) is a shortcut for ``for all but finitely many n''.
\end{definition*}

\begin{definition*}[\cite{Mokobodzki:1967}] An ultrafilter \(\mathcal U\) on \(\omega\) is \emph{rapid}
if for every \(f:\omega\to\omega\) there is an \(X\in\mathcal U\) such that the function enumerating $X$ in increasing order dominates \(\left\{f\right\}\).
To make notation simpler we will
write \(X(n)\) to denote the \(n\)-th element of \(X\) in its increasing enumeration and \(X[n]=X\setminus X(n)\).
\end{definition*}

Again we will use an alternate characterization:

\begin{fact} \label{fact:alt-rapid}
An ultrafilter \(\mathcal U\) is rapid iff for every partition \(\{K_n:n<\omega\}\)
of \(\omega\) into finite sets there is \(X\in\mathcal U\) such that \(|X\cap K_n|\leq n\)
for all \(n<\omega\) iff for every infinite \(A\subseteq\omega\) there is \(X\in\mathcal U\) such that
\(|X\cap n|\leq|A\cap n|^2\) for all \(n<\omega\).
\end{fact}

\begin{definition*}[\cite{Katetov:1968}]
The \emph{Rudin-Keisler ordering} of ultrafilters is defined as follows. Given two
ultrafilters \(\mathcal U,\mathcal V\) on \(\omega\) we say that \(\mathcal U\)
is \emph{Rudin-Keisler below} (or that it is \emph{Rudin-Keisler reducible} to)
\(\mathcal V\), denoted \(\mathcal U\leq_{RK}\mathcal V\), if there is a function
\(f:\omega\to\omega\) such that
\[
  \mathcal U = f_*(\mathcal V) = \left\{X\subseteq\omega:f^{-1}[X]\in\mathcal V \right\}.
\]
If the function is finite-to-one we say that \(\mathcal U\) is \emph{Rudin-Blass}
below \(\mathcal V\), \(\mathcal U\leq_{RB}\mathcal V\). If the function is both
finite-to-one \emph{and nondecreasing} we write \(\mathcal U\leq_{RB}^+\mathcal V\).
\end{definition*}
More information about the ${\leq}^{+}_{\mathrm{RB}}$ ordering on the ultrafilters can be found in \cite{Laflamme:1998}.
A major difference between the ${\leq}_{\mathrm{RB}}$ and ${\leq}^{+}_{\mathrm{RB}}$ orderings, which was discovered by Laflamme and Zhu in \cite{Laflamme:1998}, is that ${\leq}^{+}_{\mathrm{RB}}$ is a tree-like ordering.
In other words, for any ultrafilters $\UUU$, $\VVV$, and $\WWW$, if $\UUU \: {\leq}^{+}_{\mathrm{RB}} \: \WWW$ and $\VVV \: {\leq}^{+}_{\mathrm{RB}} \: \WWW$, then either $\UUU \: {\leq}^{+}_{\mathrm{RB}} \: \VVV$ or $\VVV \: {\leq}^{+}_{\mathrm{RB}} \: \UUU$.
This is very much false for the ${\leq}_{\mathrm{RB}}$ ordering even when it is restricted to the class of P-points, as was shown by Blass~\cite{Blass:1973} who constructed a P-point with two incomparable predecessors assuming $\MA$.

It is easy to see that being rapid and being a P-point is preserved when going
down in the Rudin-Keisler ordering and that the Rudin-Keisler and Rudin-Blass orderings
coincide below every P-point. Also, since Rudin-Keisler reducibility has
to be witnessed by some function \(f:\omega\to\omega\) and since two RK-inequivalent
ultrafilters can't be witnessed to be below a third by a single function it immediately
follows that every ultrafilter has at most \(\mathfrak c\)-many RK-predecessors.

Another ordering of ultrafilters is the Tukey ordering. It was introduced by
Tukey in \cite{Tukey:1940} for comparing the cofinal type of arbitrary directed partial orders. Isbell (\cite{Isbell:1965})
was the first who used the Tukey ordering to compare ultrafilters.

\begin{definition*}[\cite{Isbell:1965}]\label{def:tukey}
Let $\UUU$ and $\VVV$ be ultrafilters on $\omega$.
 We say that $\UUU \: {\leq}_{T} \: \VVV$, i.e.\@ $\UUU$ is \emph{Tukey reducible} to $\VVV$ or $\UUU$ is \emph{Tukey below} $\VVV$, if there is a map $\phi: \VVV \to \UUU$ such that $\forall A, B \in \VVV\left[A \subseteq B \implies \phi(A) \subseteq \phi(B)\right]$ and $\forall A \in \UUU \exists B \in \VVV\left[\phi(B) \subseteq A\right]$.
 We say that $\UUU \: {\equiv}_{T} \: \VVV$, i.e.\@ $\UUU$ is \emph{Tukey equivalent} to $\VVV$, if $ \UUU \: {\leq}_{T} \: \VVV$ and $\VVV \: {\leq}_{T} \: \UUU$.
\end{definition*}

Recently the interest in this ordering on ultrafilters was revived by the paper
\cite{Milovich:2008}. See also \cite{Dobrinen:2011} and \cite{Raghavan:2012}.

Finally, to eliminate some extraneous brackets, we will use the convenient
standard shorthand \(f^{-1}(n)\) to denote the preimage of \(\{n\}\) instead
of the formally more correct \(f^{-1}[\{n\}]\).

\section{There is no short unbounded chain of rapid P-points} \label{sec:noshort}

We start with a few simple observations. Below we use \(Poly\) to denote
the following set of polynomial functions\footnote{The fact that they are polynomials
is not important. We could as well have chosen all functions of some countable
elementary submodel of the universe; all that we need is that each function
grows much faster than the previous one.}: \(\{n^k:k<\omega\}\).

\begin{observation}
If  \(f:\omega\to\omega\) dominates \(Poly\) then so does
\(f^\prime(n)=\frac{f(n)}{2n}-n^2\).
\end{observation}

It is easy to see that the function \(n\) in \autoref{fact:alt-rapid} could as
well have been replaced by any function tending to infinity:

\begin{observation}\label{rapid-growth}
If  \(s:\omega\to\omega\) is a function tending to infinity
(i.e. \(\lim\inf s(n)=\infty\)), \(\pi:\omega\to\omega\) is a finite to one function,
and \(\mathcal U\) is a rapid ultrafilter then there is \(X\in\mathcal U\)
such that \((\forall n<\omega)(|\pre{\pi}{n}\cap X|\leq s(n))\)
\end{observation}

We aim to show that each RK-increasing chain of rapid P-points of length
 \(<\omega_2\) has a rapid P-point on top. We do this by taking the chain and
recursively constructing the future projections (called \(g\) in the following
proposition) from the top to each ultrafilter in the sequence. If these
projections commute with each of the maps witnessing the RK-relations in the
chain, then the inverse images of the chain by these projections will generate
a P-filter. By a relatively easy argument we can guarantee that it will be an
ultrafilter (making sure that at each step we decide one set). To make it rapid
we have to work more. For this purpose, we will also build a tower on the side
(the \(T\)s in the following proposition), which will generate a rapid P-point
and, moreover, this P-point will be compatible with the final P-filter.

The following proposition gives a single step of the construction. The set \(A\) in the
assumption will be later used to make sure that the top filter is both an ultrafilter
and rapid. The key property that will keep the induction going will be the fact
that the \(g_\alpha\)s are finite to one but \emph{not} bounded-to-one in a very
strong sense: the size of the preimages of points (we will call this somewhat imprecisely the \emph{growth rate}
of \(g\)) will dominate a function \(s\) which will in turn dominate
the set \(Poly\) (conditions (1\&2)). Moreover the first part of condition (4)
will guarantee that the maps \(g_\alpha\) will not be bijections on some large set
(otherwise our supposed upper bound would be RK-equivalent to some \(\mathcal U_\alpha\)).

We use the following conventions: we imagine each \(\mathcal U_\alpha\) lives on a separate copy
of \(\omega\) (the \(\alpha\)-the level). We will use the letter \(m\) to denote numbers on the first
level (i.e. where \(\mathcal U_0\) lives), the letter \(n\) will denote numbers
living on some level \(0<\alpha<\delta\), the letter \(l\) will denote numbers
living on the final level (i.e. where the top ultrafilter we will be constructing
lives) and the letter \(k\) will denote numbers living on level \(\delta\). The
letters \(i\) and \(j\) will be used as unrelated natural numbers. If a function
has two ordinal indices \(\alpha\beta\), they indicate that it goes from the \(\alpha\)-th level
down to the \(\beta\)-th level. Finally, the functions \(g_\alpha\) go from the final level to the
level indicated by their ordinal index.

\begin{proposition}
Assume  \(\delta<\omega_1\) and
\(\langle \mathcal U_\alpha:\alpha\leq\delta\rangle\) is an RK-increasing sequence of
rapid P-points as witnessed by finite-to-one maps
\(\Pi=\{\pi_{\alpha\beta}:\beta\leq\alpha\leq\delta\}\) with
\(\pi_{\alpha\alpha}=Id\) for each \(\alpha\leq\delta\). Also, let
\(\bar{s}=\langle s_\alpha:\alpha<\delta\rangle\) be a sequence of maps,
each dominating \(Poly\), \(\bar{g}=\langle g_\alpha:\alpha<\delta\rangle\) a sequence of
finite-to-one maps, \(\bar{T}=\langle T_\alpha:\alpha<\delta\rangle\) a
\(\subseteq^*\)-decreasing sequence of subsets of \(\omega\),
and \(A\subseteq\omega\). Suppose, moreover, that the following conditions
are satisfied:

\begin{enumerate}[series=mainprop]
  \item \label{cond:growth} the growth rate of \(g_\alpha\) dominates \(s_\alpha\circ\pi_{\alpha0}\), i.e. \((\forall \alpha)(\forall^\infty n)(|g_\alpha^{-1}[\{n\}]|\geq s_\alpha(\pi_{\alpha0}(n)))\);
  \item \label{cond:s-order}the sequence \(\bar{s}\) is \(<^*\) decreasing in the following (stronger) sense:
  \[
    (\forall \alpha<\beta<\delta)
    (\forall^\infty m)
    \left(s_\alpha(m)\geq \frac{s_{\alpha}(m)}{2m}-m^2\geq s_\beta(m)\right);
  \]
  \item \label{cond:commuting} \(\Pi\cup\{g_\alpha:\alpha<\delta\}\) commute, i.e. for all \(\beta\leq\alpha<\delta\) the following diagram commutes on a \(\mathcal U_\alpha\)-large set
  \[
  \begin{tikzcd}
    \omega \arrow[d, "\pi_{\alpha\beta}"'] & \omega \arrow[l, "g_\alpha"'] \arrow[ld, "g_\beta"] \\
    \omega
  \end{tikzcd}
  \]
  formally there is \(X\in\mathcal U_\alpha\) such that \(g_\beta(l)=\pi_{\alpha\beta}(g_\alpha(l))\)
  for each \(l\in g_\alpha^{-1}[X]\); and
  \item \label{cond:g-preimages} for each  \(\alpha<\delta\) there is \(X\in\mathcal U_\alpha\) such that
  \[
    \lim_{n\in X}|\pre{g_\alpha}{n}\cap T_\alpha| = \infty;
  \]
  while also
  \[
    |\pre{g_\alpha}{n}\cap T_\alpha|\leq \min \big(\pre{g_\alpha}{n}\cup\{\pi_{\alpha0}(n)\}\big),
  \]
  for each \(n\in X\).
\end{enumerate}

Then we can extend the sequences  \(\bar{g},\bar{s},\bar{T}\) by constructing
the maps \(g_\delta\) and \(s_\delta\) and a set \(T_\delta\) so that
(corresponding modifications of) (1-4) are still satisfied and, moreover,
\((\forall i)(|T_\delta\cap i|\leq |A\cap i|^2)\) and \(T_\delta\) decides \(A\)
(i.e. \(T_\delta\subseteq A\) or \(T_\delta\cap A=\emptyset\)).
\end{proposition}

\begin{proof}
We first introduce some notation. Fix  \(D\subseteq\delta\) a cofinal subset of
\(\delta\) of order type \(\omega\) such that \(0\in D\). In our construction we will only deal with
\(\alpha\in D\). Given \(\alpha\in D\) we write
\[
  \alpha^+=\min \big(\left\{\beta\in D:\alpha<\beta\right\}\big)
\]
for the successor of \(\alpha\) in \(D\). We also let
\[
 \ord{\alpha}=|D\cap\alpha|,
\]
i.e. \(\alpha\) is the \(\ord{\alpha}\)-th element of \(D\). Next we use \(D\) to
enumerate \(Poly\) in an increasing sequence: \(Poly=\{p_\alpha:\alpha\in D\}\),
where \(p_\alpha\leq p_{\alpha^+}\).

Given  \(\alpha\in D\) let \(c_n^\alpha=\pre{\pi_{\delta\alpha}}{n}\).
Since \(\mathcal U_\delta\) is rapid we can use \autoref{fact:alt-rapid} to
find \(X_\alpha\in\mathcal U_\delta\) such that \(|X_\alpha\cap c^\alpha_n|\leq n\)
for each \(n<\omega\).
We can also assume that\footnote{otherwise throw finitely many elements of \(X_\alpha\) away
to get the first requirement and for the second intersect it with the set \(\pi_{\delta\alpha}^{-1}[X]\),
where \(X\) is the set guaranteed to exist by condition (4) above.}
\[
  \tag{$*$}\label{eq:growth-rate-g}
  |\pre{g_\alpha}{n}|\geq s_\alpha(\pi_{\alpha0}(n))
\]
for all \(n\) such that \(X_\alpha\cap c^\alpha_n\neq\emptyset\) and that
if \(n\in \pi_{\delta\alpha}[X_\alpha]\) then
\[
  \tag{$\dag$}\label{eq:g-bound}
  |\pre{g_\alpha}{n}\cap T_\alpha|\leq \min \big(\pre{g_\alpha}{n}\cup\{\pi_{\alpha0}(n)\}\big)
\]
Next, choose
\(Y_\alpha\in\mathcal U_\delta\) such that
\[
  \{g_\beta:\beta\in D\ \&\ \beta\leq\alpha^+\}\cup
  \{\pi_{\gamma\beta}:\gamma\leq\beta\leq\alpha^+,\gamma,\beta\in D\}
\]
commute on \(Y_\alpha\); more precisely, for every \(\beta\leq\gamma\leq\alpha^+\)
all in \(D\) and every \(k\in Y_\alpha\) and any \(l\in g_\gamma^{-1}[\pi_{\delta\gamma}(k)]\)
we have
\[
  \tag{$\ddag$}\label{eq:commuting}
  \pi_{\gamma\beta}(\pi_{\delta\gamma}(k))=\pi_{\delta\beta}(k)=g_\beta(l).
\]

Finally, since  \(g_\alpha\) and \(T_\alpha\) satisfy the first part of (4), we can use
\autoref{rapid-growth} to find \(Z_\alpha\in \mathcal U_\delta\) such that
\[\tag{$\mathsection$}\label{eq:pi-bound}
  (\forall n)\Big(
    |\pre{\pi_{\delta\alpha}}{n}\cap Z_\alpha|\leq
    \frac{|\pre{g_\alpha}{n}\cap T_\alpha|}{\ord{\alpha}}
  \Big),
\]
i.e. \(Z_\alpha\) is \(\ord{\alpha}\)-times more sparse then \(T_\alpha\).
(Just apply \autoref{rapid-growth} to
\(\pi=\pi_{\delta\alpha}\), \(\mathcal U=\mathcal U_\delta\),
and \(s(n)=|\pre{g_\alpha}{n}\cap T_\alpha|/\ord{\alpha}\).)

Since  \(\mathcal U_\delta\) is a P-point, we can find \(X\in\mathcal U_\delta\)
which is a pseudointersection of \(\{X_\alpha,Y_\alpha,Z_\alpha:\alpha\in D\}\).
Recursively construct a partition \(\{K_\alpha:\alpha\in D\}\) of \(X\)  into finite
sets and  \(s_\delta:\omega\to\omega\) such that

\[
  K_\alpha\subseteq\bigcap_{\beta\in D\cap\alpha^+} X_\beta\cap Y_\beta\cap Z_\beta,
\]

and

\begin{enumerate}[resume=mainprop]
  \item \label{eq:s-bound}\(s_\alpha(m)\geq s_{\alpha}(m)/m-m^2\geq s_\delta(m)\geq p_\alpha(m)\)
	     whenever \(m=\pi_{\delta0}(n)\) and \(n\in K_\alpha\);
  \item \label{eq:pi-closure} \(\pi_{\delta\alpha}^{-1}[\pi_{\delta\alpha}[K_\alpha]]\cap X\subseteq{K_\alpha}\); and
  \item \label{eq:ord-bound} \(\ord{\alpha}\leq \big|A\cap \min\big(\pi_{\delta0}[K_\alpha]\big)\cap\min \big(g_\alpha^{-1}[\pi_{\delta\alpha}[K_\alpha]]\big)\big|\).
\end{enumerate}

This is not hard to do: first find an increasing sequence \(\{k_\alpha:\alpha\in D\}\)
of natural numbers such that
\[
    X\setminus k_\alpha\subseteq \bigcap_{\beta\in D\cap\alpha^+} X_\beta\cap Y_\beta\cap Z_\beta;
\]
and
\[
  s_\alpha(m)\geq \frac{s_{\alpha}(m)}{m}-m^2\geq p_{\alpha^+}(m)\quad\&\quad \ord{\alpha}\leq |A\cap m\cap i|
\]
for all \(m = \pi_{\delta0}(n)\) and \(i\in g_{\alpha}^{-1}(\pi_{\delta\alpha}(n))\) with \(n\in X\setminus k_\alpha\). Then let
\[
K_\alpha=\pi_{\delta\alpha}^{-1}[\pi_{\delta\alpha}[X\cap [k_\alpha,k_{\alpha^+})]
\]
and
\[
 s_\delta\upharpoonright \pi_{\delta0}[K_\alpha]=p_\alpha.
\]
(Formally, this won't be a partition, since it will not cover \(X\cap[0,k_0)\); we can
just throw these finitely many elements out of \(X\)).

Let  \(J_\alpha=\pi_{\delta\alpha}[K_\alpha]\) and \(L_\alpha=g_\alpha^{-1}[J_\alpha]\).
Notice that since \(K_\alpha\subseteq Y_\alpha\), we can use \eqref{eq:commuting} and \eqref{eq:pi-closure} to conclude that
\(L_\alpha\cap L_\beta=\emptyset\) for distinct \(\alpha\neq\beta\in D\). This allows us to define \(g_\delta\) separately
on each \(L_\alpha\). For \(n\in J_\alpha\) let \(b_n^\alpha=\pre{g_\alpha}{n}\).

\begin{center}
\def\svgwidth{8cm}
%% Creator: Inkscape inkscape 0.92.3, www.inkscape.org
%% PDF/EPS/PS + LaTeX output extension by Johan Engelen, 2010
%% Accompanies image file 'diagram-image.pdf' (pdf, eps, ps)
%%
%% To include the image in your LaTeX document, write
%%   \input{<filename>.pdf_tex}
%%  instead of
%%   \includegraphics{<filename>.pdf}
%% To scale the image, write
%%   \def\svgwidth{<desired width>}
%%   \input{<filename>.pdf_tex}
%%  instead of
%%   \includegraphics[width=<desired width>]{<filename>.pdf}
%%
%% Images with a different path to the parent latex file can
%% be accessed with the `import' package (which may need to be
%% installed) using
%%   \usepackage{import}
%% in the preamble, and then including the image with
%%   \import{<path to file>}{<filename>.pdf_tex}
%% Alternatively, one can specify
%%   \graphicspath{{<path to file>/}}
%% 
%% For more information, please see info/svg-inkscape on CTAN:
%%   http://tug.ctan.org/tex-archive/info/svg-inkscape
%%
\begingroup%
  \makeatletter%
  \providecommand\color[2][]{%
    \errmessage{(Inkscape) Color is used for the text in Inkscape, but the package 'color.sty' is not loaded}%
    \renewcommand\color[2][]{}%
  }%
  \providecommand\transparent[1]{%
    \errmessage{(Inkscape) Transparency is used (non-zero) for the text in Inkscape, but the package 'transparent.sty' is not loaded}%
    \renewcommand\transparent[1]{}%
  }%
  \providecommand\rotatebox[2]{#2}%
  \newcommand*\fsize{\dimexpr\f@size pt\relax}%
  \newcommand*\lineheight[1]{\fontsize{\fsize}{#1\fsize}\selectfont}%
  \ifx\svgwidth\undefined%
    \setlength{\unitlength}{472.16699315bp}%
    \ifx\svgscale\undefined%
      \relax%
    \else%
      \setlength{\unitlength}{\unitlength * \real{\svgscale}}%
    \fi%
  \else%
    \setlength{\unitlength}{\svgwidth}%
  \fi%
  \global\let\svgwidth\undefined%
  \global\let\svgscale\undefined%
  \makeatother%
  \begin{picture}(1,1.09418886)%
    \lineheight{1}%
    \setlength\tabcolsep{0pt}%
    \put(0,0){\includegraphics[width=\unitlength,page=1]{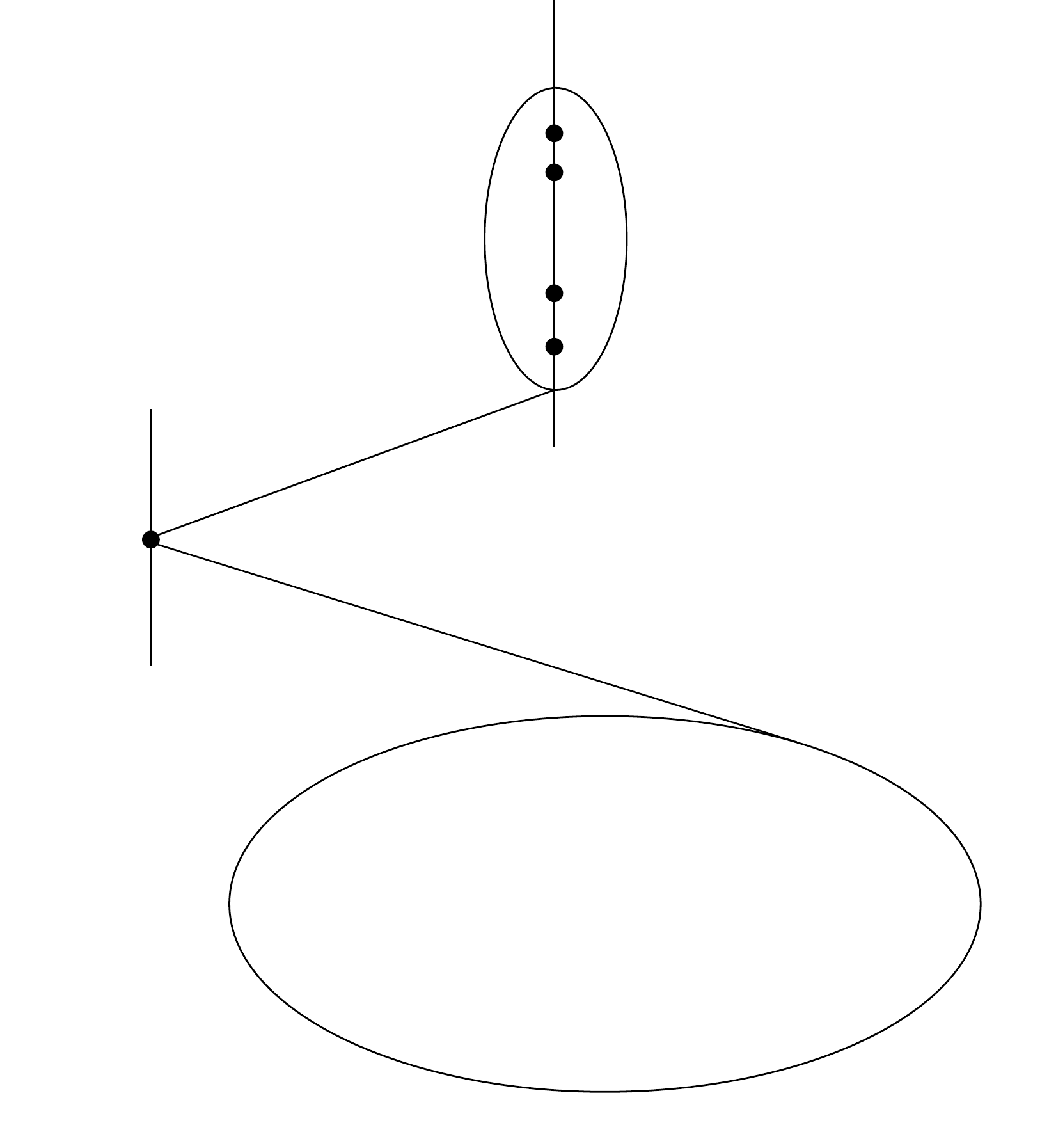}}%
    \put(0.79003807,0.4531134){\color[rgb]{0,0,0}\makebox(0,0)[lt]{\lineheight{0}\smash{\begin{tabular}[t]{l}$b_n^\alpha$\end{tabular}}}}%
    \put(0.09327856,0.58871146){\color[rgb]{0,0,0}\makebox(0,0)[lt]{\lineheight{0}\smash{\begin{tabular}[t]{l}$n$\end{tabular}}}}%
    \put(0.53869443,0.75290525){\color[rgb]{0,0,0}\makebox(0,0)[lt]{\lineheight{0}\smash{\begin{tabular}[t]{l}$k_0$\end{tabular}}}}%
    \put(0.26760391,0.83423241){\color[rgb]{0,0,0}\makebox(0,0)[lt]{\lineheight{0}\smash{\begin{tabular}[t]{l}$\pi_{\delta\alpha}$\end{tabular}}}}%
    \put(0.25631136,0.54828873){\color[rgb]{0,0,0}\makebox(0,0)[lt]{\lineheight{0}\smash{\begin{tabular}[t]{l}$g_\alpha$\end{tabular}}}}%
    \put(0.08994275,0.68610082){\color[rgb]{0,0,0}\makebox(0,0)[lt]{\lineheight{0}\smash{\begin{tabular}[t]{l}$J_\alpha$\end{tabular}}}}%
    \put(0.4728581,1.0641753){\color[rgb]{0,0,0}\makebox(0,0)[lt]{\lineheight{0}\smash{\begin{tabular}[t]{l}$K_\alpha$\end{tabular}}}}%
    \put(0,0){\includegraphics[width=\unitlength,page=2]{diagram-image.pdf}}%
    \put(0.53869443,0.80712336){\color[rgb]{0,0,0}\makebox(0,0)[lt]{\lineheight{0}\smash{\begin{tabular}[t]{l}$k_1$\end{tabular}}}}%
    \put(0,0){\includegraphics[width=\unitlength,page=3]{diagram-image.pdf}}%
    \put(0.25332764,0.17404106){\color[rgb]{0,0,0}\makebox(0,0)[lt]{\lineheight{0}\smash{\begin{tabular}[t]{l}$e^n_{k_0}$\end{tabular}}}}%
    \put(0.40051814,0.09702281){\color[rgb]{0,0,0}\makebox(0,0)[lt]{\lineheight{0}\smash{\begin{tabular}[t]{l}$e^n_{k_1}$\end{tabular}}}}%
    \put(0.6170251,0.93823359){\color[rgb]{0,0,0}\makebox(0,0)[lt]{\lineheight{0}\smash{\begin{tabular}[t]{l}$\leq m$\end{tabular}}}}%
    \put(0.78645668,0.03346901){\color[rgb]{0,0,0}\makebox(0,0)[lt]{\lineheight{0}\smash{\begin{tabular}[t]{l}$\geq s_\alpha(m)$\end{tabular}}}}%
    \put(0,0){\includegraphics[width=\unitlength,page=4]{diagram-image.pdf}}%
    \put(0.56806549,0.56184325){\color[rgb]{0,0,0}\makebox(0,0)[lt]{\lineheight{0}\smash{\begin{tabular}[t]{l}$g_\delta$\end{tabular}}}}%
    \put(0,0){\includegraphics[width=\unitlength,page=5]{diagram-image.pdf}}%
    \put(0.08349115,0.23314594){\color[rgb]{0,0,0}\makebox(0,0)[lt]{\lineheight{0}\smash{\begin{tabular}[t]{l}$\min b_\alpha^n$\end{tabular}}}}%
    \put(0,0){\includegraphics[width=\unitlength,page=6]{diagram-image.pdf}}%
    \put(-0.00149991,0.47082316){\color[rgb]{0,0,0}\makebox(0,0)[lt]{\lineheight{0}\smash{\begin{tabular}[t]{l}$m$\end{tabular}}}}%
    \put(0,0){\includegraphics[width=\unitlength,page=7]{diagram-image.pdf}}%
    \put(0.27704806,0.27570001){\color[rgb]{0,0,0}\makebox(0,0)[lt]{\lineheight{0}\smash{\begin{tabular}[t]{l}$d^n_{k_0}$\end{tabular}}}}%
    \put(0.45473625,0.27323165){\color[rgb]{0,0,0}\makebox(0,0)[lt]{\lineheight{0}\smash{\begin{tabular}[t]{l}$d^n_{k_1}$\end{tabular}}}}%
    \put(0.02285693,0.52520466){\color[rgb]{0,0,0}\makebox(0,0)[lt]{\lineheight{0}\smash{\begin{tabular}[t]{l}$\pi_{\alpha0}$\end{tabular}}}}%
    \put(0,0){\includegraphics[width=\unitlength,page=8]{diagram-image.pdf}}%
  \end{picture}%
\endgroup%

\end{center}

Fix \(\alpha\in D\) and \(n\in J_\alpha\) and let \(m=\pi_{\alpha0}(n)\). Then, since \(K_\alpha\subseteq X_\alpha\), by \eqref{eq:growth-rate-g} and \eqref{eq:s-bound} we have
\(|b_n^\alpha|\geq ms_\delta(m)+m^3\). Moreover, since \(K_\alpha\subseteq Z_\alpha\),
by \eqref{eq:pi-bound} and \eqref{eq:g-bound} we have
\[
  |K_\alpha\cap\pre{\pi_{\delta\alpha}}{n}| \leq m,
\]
and also
\[
 |K_\alpha\cap\pre{\pi_{\delta\alpha}}{n}|\leq \frac{|b_n^\alpha\cap T_\alpha|}{\ord{a}}\leq\frac{\min(b_n^\alpha\cup\{m\})}{\ord{a}}
\]

It follows that we can partition \(b^\alpha_n\) into pieces \(\{e^n_k:k\in K_\alpha\cap\pre{\pi_{\delta\alpha}}{n}\}\)
each of size \(\geq s_\delta(m)+m^2\) which, moreover, satisfy
\[
  \ord{\alpha}\leq|e^n_k\cap T_\alpha|\leq \min \big(b^\alpha_n\cup\{m\}\big).
\]

Due to \eqref{eq:ord-bound} we can shrink \(e^n_k\) to a smaller set \(d^n_k\) (throwing away at most \(m\)-many elements of
\(e^n_k\cap T_\alpha\)) such that
\[
  \tag{$\mathparagraph$}\label{eq:d-bound}
  \ord{\alpha}\leq |d^n_k\cap T_\alpha|\leq |A\cap m\cap\min b^\alpha_n|
\]
Since we threw away at most \(m\) elements from each \(e^n_k\), we still have
\(|d^n_k|\geq s_\delta(m)\). Now let \(g_\delta[d^n_k]=\{k\}\) and extend  \(g_\delta\) to all of \(\omega\) arbitrarily so
that the new values are outside of \(X\) and that the requirements on \(g_\delta\) are satisfied
(i.e. that it is finite-to-one, its growth rate is bounded below by \(s_\delta\), etc.). This finishes
the construction of \(g_\delta\). For future reference, let us note that
\(g_\delta\upharpoonright (T_\alpha\cap l)\) is at most \(|A\cap l|\)-to-one
for any \(l<\omega\).

Notice that if \(\beta<\alpha\in D\), \(k\in K_\alpha\) and \(g_\delta(l)=k\) then
\(g_\beta(l)=\pi_{\alpha\beta}(g_\alpha(l))\) and so
\[
  \pi_{\delta\beta}(k)=\pi_{\alpha\beta}(\pi_{\delta\alpha}(k))=g_\beta(l)
\]
so \eqref{cond:commuting} is satisfied for \(g_\delta\).
That condition \eqref{cond:s-order} for \(s_\delta\) is satisfied follows from \eqref{eq:s-bound}.
That condition \eqref{cond:growth} is satisfied for \(s_\delta\) and \(g_\delta\) follows from
the construction (\(|d^n_k|\geq s_\delta(m)\)).

(Formally, we have only checked the conditions for \(\beta,\alpha\in D\), but
this is clearly enough, since \(D\) is cofinal in \(\delta\).)

Finally we must construct  \(T_\delta\). Without loss of generality we may assume
that \(T_\alpha\subseteq T_\beta\) for \(\beta<\alpha\in D\) (otherwise we
could have carried out the construction for some finite modifications of
\(T_\alpha\)s and the resulting \(T_\delta\) would still work for the
original \(T_\alpha\)s). Let
\[
  T^\prime = \bigcup_{\alpha\in D}g^{-1}_\delta[K_\alpha]\cap T_\alpha.
\]
Then \(T^\prime\) is a pseudointersection of \(\{T_\alpha:\alpha\in D\}\).
Moreover \(g_\delta\) was constructed so that (see \eqref{eq:d-bound})
\[
  \ord{\alpha}\leq
  |\pre{g_\delta}{k}\cap T_\alpha|\leq
  \min\{m\}\cup \pre{g_\delta}{k}
\]
for each \(k\in K_\alpha\) and \(m=\pi_{\delta0}(k)\). It follows that \(T^\prime\) and \(g_\delta\)
satisfy (the corresponding modification of) \eqref{cond:g-preimages} and, in particular, that
\(g_\delta[T^\prime]\in\mathcal U_\delta\).

\begin{claim*} There is an \(X^\prime\in\mathcal U_\delta\), \(X^\prime\subseteq X\cap g_\delta[T^\prime]\) such that
\[
  |X^\prime\cap g_\delta[l]|\leq|A\cap l|,
\]
\end{claim*}
\begin{proof}[Proof of claim] Since \(|g_\delta[l]|\leq l\) we can find a bijection \(\pi:\omega\to\omega\)
such that \(\pi[g_\delta[l]]\subseteq l\). Since \(\pi\) is a bijection, the ultrafilter
\(\pi^*(\mathcal U_\delta)\) is rapid so there is \(Y\in\pi^*(\mathcal U_\delta)\) such that
\[
    |Y\cap l|\leq|A\cap l|.
\]
Let \(X^\prime =\pi^{-1}[Y]\cap X\cap g_\delta[T^\prime]\in\mathcal U_\delta\). Then we have
\[
X^\prime\cap g_\delta[l]\subseteq X^\prime\cap\pi^{-1}[l]\subseteq\pi^{-1}[Y]\cap\pi^{-1}[l]=\pi^{-1}[Y\cap l].
\]
Using this, the fact that \(\pi\) is a bijection and the choice of Y, we get
\[
    |X^\prime\cap g_\delta[l]|\leq|\pi^{-1}[Y\cap l]|=|Y\cap l|\leq |A\cap l|.
\]
which finishes the proof of the claim.
\end{proof}

Let \(T^{\prime\prime}=g_{\delta}^{-1}[X^\prime]\subseteq T^\prime\). Now, since
\(g_\delta\upharpoonright T^\prime\cap l\) is at most \(|A\cap l|\)-to-1,
we have
\[
  |T^{\prime\prime}\cap l|
  \leq |A\cap l|\cdot|X^\prime\cap g_\delta[l]|
  \leq |A\cap l|\cdot |A\cap l|\leq |A\cap l|^2
\]
Finally notice that
\[
    g_\delta[T^{\prime\prime}\cap A]\cup g_\delta[T^{\prime\prime}\setminus A]=
    g_\delta[T^{\prime\prime}]=X^\prime\in\mathcal U_\delta.
\]
So, since \(\mathcal U_\delta\) is an ultrafilter, we can choose
\(T_\delta\subseteq T^{\prime\prime}\) which decides \(A\) and still satisfies
\eqref{cond:g-preimages}. This finishes the proof.
\end{proof}

\begin{theorem}\label{thm:closed}
Assume CH. Every RK-increasing chain of rapid P-points of length  \(\omega_1\) has
an upper bound which is also a rapid P-point.
\end{theorem}

\begin{proof}
Let  \(\langle\mathcal U_\alpha:\alpha<\omega_1\rangle\) be an
RK-increasing chain of rapid P-points as witnessed by finite-to-one maps
\(\Pi=\{\pi_{\alpha\beta}:\beta\leq\alpha<\omega_1\}\). Without loss of
generality we may assume that the maps commute in the sense of \eqref{cond:commuting}
in the previous proposition. Enumerate \([\omega]^\omega\) as \(\{A_\alpha:\alpha<\omega_1\}\). Recursively
build a sequence of finite-to-one maps \(\langle g_\alpha:\alpha<\omega_1\rangle\) and
a decreasing tower \(\langle T_\alpha:\alpha<\omega_1\rangle\) so that
\(T_\alpha\) decides \(A_\alpha\) and \(|T_\alpha\cap n|\leq |A_\alpha\cap n|^2\).
This can be done by repeatedly applying the previous proposition at each step.
In the end \(\langle T_\alpha:\alpha<\omega_1\rangle\) generates a rapid P-point and
the map \(g_\alpha\) witnesses that this P-point is above \(\mathcal U_\alpha\).
\end{proof}

We do not know the optimal hypothesis needed to carry out the above proof.
We leave it as a question for further research.
\begin{question} What is the optimal hypothesis needed to get the conclusion of Theorem \ref{thm:closed} with ${\omega}_{1}$ replaced by $\mathfrak{c}$.
In particular, does this hold if we replace CH by \(\mathfrak b=\mathfrak c\)?
Or even \(\mathfrak d=\mathfrak c\)?
\end{question}

\section{A short unbounded chain of P-points} \label{sec:short}
In this section we show, assuming  \(\Diamond\), that there is an RK-chain of
P-points of length $\omega_1$ which has no P-point RK-above. We assume \(\Diamond\) only for simplicity;
a more involved argument using the Devlin-Shelah weak diamond can be used to
construct the chain, e.g., under CH.

\begin{definition*}
Let  \(\bar{\mathcal U}=\langle \mathcal U_\alpha:\alpha<\delta\rangle\) be a
sequence of ultrafilters and \(\Pi=\langle\pi_{\alpha\beta}:\beta\leq\alpha\leq\delta\rangle\)
a family of maps from \(\omega\) to \(\omega\). We say that \emph{\(\Pi\) commutes
with respect to \(\bar{\mathcal U}\)}, if for \(\beta\leq\alpha\leq\gamma\leq\delta\)
there is \(X\in\mathcal U_\gamma\) such
that \((\forall i\in X)(\pi_{\alpha\beta}(\pi_{\gamma\alpha}(i))=\pi_{\gamma\beta}(i))\).
When the sequence \(\bar{\mathcal U}\) is clear from the context, we just say
that \(\Pi\) commutes.
\end{definition*}
\begin{notation*}
Given two families \(\Pi_i=\langle\pi^i_{\delta\alpha}:\alpha<\delta\rangle,i<2\)
of maps and a sequence of ultrafilters \(\bar{\mathcal U}\) as above, we write \(f:\Pi_0\to_{\bar{\mathcal U}}\Pi_1\) to indicate that \(f\) is a map from
\(\omega\) to \(\omega\) and for each \(\alpha<\delta\) there is an \(X\in(\pi^0_{\delta\alpha})^{-1}[\mathcal U_\alpha]\)
such that \(\pi^0_{\delta\alpha}(n)=\pi^1_{\delta\alpha}(f(n))\) for all \(n\in X\).
\end{notation*}

\begin{definition*}
Given two families of maps
 \(\Pi_i=\langle\pi^i_{\delta\alpha}:\alpha<\delta\rangle,i<2\)
we say that \(\Pi_1\prec \Pi_0\) if
\[
(\forall \alpha<\delta)(\forall^\infty n<\omega)\big(
  |\pre{(\pi^1_{\delta\alpha})}{n}| > n\cdot |\pre{(\pi^0_{\delta\alpha})}{n}|
\big)
\]
Moreover, if \(\pi_0,\pi_1\) are two maps, \(\mathcal U\) is an ultrafilter,
and \(X\in[\omega]^\omega\), we write
\[
  \pi_1\prec_{X,\mathcal U} \pi_0
\]
if there is a \(Y\in\mathcal U\) and \(s:\omega\to\omega\) tending to infinity
such that
\[
  (\forall n\in Y)\big(|\pre{\pi_1}{n}\cap X|>s(n)\cdot|\pre{\pi_0}{n}\cap X|\big).
\]
\end{definition*}

\begin{observation}
Assume  \(\bar{\mathcal U}=\langle U_{\alpha}:\alpha<\delta\rangle\) \label{refine_maps}
is an
RK-increasing chain of P-points of length \(\delta<\omega_1\) as witnessed
by a family of finite-to-one maps \(\Pi=\langle\pi_{\alpha\beta}:\beta\leq\alpha<\delta\rangle\).
Suppose, moreover, that we are given
a family of finite-to-one maps \(\Pi_0=\langle \pi^0_{\delta\alpha}:\alpha<\delta\rangle\)
such that \(\Pi\cup\Pi_0\) commute. Then there is a family
\(\Pi_1=\langle \pi^1_{\delta\alpha}:\alpha<\delta\rangle\) such that \(\Pi_1\prec \Pi_0\) and \(\Pi\cup \Pi_1\) still commutes.
\end{observation}

\begin{proof}
Fix an arbitrary finite-to-one  \(\pi\) such that
\(|\pre{\pi}{n}|\geq n\) and let \(\pi_{\delta\alpha}^1(n)=\pi_{\delta\alpha}^0(\pi(n))\).
\end{proof}

\begin{definition*}
Given an RK-increasing chain of P-points  \(\bar{\mathcal U}\) of
length \(\delta\) for some limit \(\delta<\omega_1\) a
family of finite-to-one maps \(\Pi=\langle\pi_{\alpha\beta}:\beta\leq\alpha<\delta\rangle\) witnessing that the
chain is RK-increasing and two families of finite-to-one maps \(\Pi_i=\langle\pi^i_{\delta\alpha}:\alpha<\delta\rangle,i<2\), such that
\(\Pi_1\prec\Pi_0\) and \(\Pi\cup\Pi_i\) commutes w.r.t. \(\bar{\mathcal U}\) for \(i<2\), we define the forcing
\[
\mathbb P(\bar{\mathcal U},\Pi,\Pi_0,\Pi_1)=
\bigg(\big\{
  X\in[\omega]^\omega:
  (\forall \alpha<\delta)
  (\pi_{\delta\alpha}^1\prec_{X,\mathcal U_\alpha}\pi_{\delta\alpha}^0)\big\},
  \subseteq^*\bigg)
\]
\end{definition*}

For the following observation and propositions, fix \(\delta\), \(\bar{\mathcal U}\), \(\Pi,\Pi_0\), and \(\Pi_1\)
as in the definition.

\begin{observation}
The forcing  \(\mathbb P(\bar{\mathcal U},\Pi,\Pi_0,\Pi_1)\) contains
\(\omega\).
\end{observation}

\begin{proposition}
The forcing  \(\mathbb P(\bar{\mathcal U},\Pi,\Pi_0,\Pi_1)\) is     \label{closed}
\(\sigma\)-closed.
\end{proposition}

\begin{proof}
Let  \(\langle X_n:n<\omega\rangle\) be a descending sequence of
conditions and, without loss of generality, assume \(X_{n+1}\subseteq X_n\) for
all \(n<\omega\). Fix a cofinal subset \(D\subseteq\delta\) of order type
\(\omega\) and, as before, write \(\alpha^+=\min\big(D\setminus(\alpha+1)\big)\) and
\(\ord{\alpha}=|D\cap \alpha|\). For each \(n<\omega\)
and \(\alpha\in D\) fix \(Y^\alpha_n\in\mathcal U_\alpha\) and \(s^\alpha_n\)
witnessing \(\pi_{\delta\alpha}^1\prec_{X_n,\mathcal U_\alpha}\pi_{\delta\alpha}^0\).
Then for each \(\alpha\in D\) let \(Y^\alpha\in\mathcal U_\alpha\) be a
pseudointersection of \(\{Y^\alpha_n:n<\omega\}\subseteq\mathcal U_\alpha\)
and fix a function \(s:\omega\to\omega\), tending to infinity, such that
\(s\leq^* s^\alpha_n\) for all \(\alpha\in D,n<\omega\).
For each \(\alpha\in D\) fix \(m_\alpha<\omega\) such that
\[
  \big|\pre{(\pi_{\delta\beta}^1)}{m}\cap X_{\ord{\alpha}}\big|\geq
  s(m)\cdot\big|\pre{(\pi_{\delta\beta}^0)}{m}\cap X_{\ord{\alpha}}\big|,
\]
for each \(\beta\in D\cap\alpha^+\) and
\(m\in Y^\alpha\setminus m_\alpha\). We also choose each \(m_\alpha\) large
enough to make sure that
\[
  \tag{$*$}\label{eq:distance}
  \max\big((\pi^i_{\delta\beta^\prime})^{-1}[m^\prime]\big) <
  \min\big((\pi^j_{\delta\beta})^{-1}(m)\big)
\]
for each \(\beta,\beta^\prime\in D\cap\alpha^+\), \(i,j<2\) and
\(m^\prime<m_\alpha<m_{\alpha+}<m\). For \(\alpha\in D\) let
\[
  X^\alpha = \bigcup_{\beta\in D\cap\alpha^+}
        (\pi^1_{\delta\beta})^{-1}
        [m_\alpha,m_{\alpha^+})\cap X_{\ord{\alpha}}
\]
Next let \(D_0\) and \(D_1\) be the set of even and odd elements of \(D\), respectively. Choose a cofinal
\(D^\prime\subseteq D\) and \(i<2\) so that
\[
  Z^\alpha=\bigcup_{\beta\in D_i}[m_\beta,m_{\beta^+})\cap Y^\alpha\in\mathcal U_\alpha.
\]
Finally, define
\[
  X=\bigcup_{\alpha\in D_i}X^\alpha.
\]
By \eqref{eq:distance} it is clear that
\[
  \tag{$\dag$}\label{eq:pseudoequality}
 X\cap(\pi^j_{\delta\beta})^{-1}[m_\alpha,m_{\alpha^+})=
 X_{\ord{\alpha}}\cap(\pi^j_{\delta\beta})^{-1}[m_\alpha,m_{\alpha^+})
\]
for each \(\alpha\in D_i\), \(\beta<\alpha\), and \(j<2\).
It is clear that \(X\) is a pseudointersection of the \(X_n\)s. We need to show
that \(X\in \mathbb P(\bar{\mathcal U},\Pi,\Pi_0,\Pi_1)\), i.e. that for each
\(\alpha<\delta\) we have
\[
  \tag{$\ddag$}\label{eq:leq}
  \pi_{\delta\alpha}^1\prec_{X,\mathcal U_\alpha}\pi_{\delta\alpha}^0.
\]
First assume \(\alpha\in D^\prime\). We show that \(Z^\alpha\setminus m_\alpha\)
witnesses \eqref{eq:leq}. Let \(m\in Z^\alpha\setminus m_\alpha\) be arbitrary. Find
\(\alpha^\prime\in D_i\) so that \(m\in[m_{\alpha^\prime},m_{\alpha^{\prime+}})\).
Then \(\alpha<\alpha^\prime\) so, in particular, we have
\[
  \big|\pre{(\pi_{\delta\alpha}^1)}{m}\cap X_{\#\alpha^\prime}\big|\geq
  s(m)\cdot\big|\pre{(\pi_{\delta\alpha}^0)}{m}\cap X_{\#\alpha^\prime}\big|
\]
This, together with \eqref{eq:pseudoequality}, shows \eqref{eq:leq}. Finally notice that
if \(\alpha<\beta<\alpha^\prime\) and
\[
  \pi_{\delta\alpha}^1\prec_{X,\mathcal U_\alpha}\pi_{\delta\alpha}^0\ \&\
  \pi_{\delta\alpha^\prime}^1\prec_{X,\mathcal U_{\alpha^\prime}}\pi_{\delta\alpha^\prime}^0
\]
then also
\[
 \pi_{\delta\beta}^1\prec_{X,\mathcal U_\beta}\pi_{\delta\beta}^0.
\]
Since \(D^\prime\) was cofinal in \(\delta\) this finishes the proof of \eqref{eq:leq}
for all \(\alpha<\delta\).
\end{proof}

\begin{proposition}
If  \(A\subseteq\omega\) then the set
\[
  D_A=\big\{X\in \mathbb P(\bar{\mathcal U},\Pi,\Pi_0,\Pi_1):X\subseteq A\vee X\subseteq\omega\setminus A\big\}
\]
is dense.
\end{proposition}

\begin{proof}
Notice that if
 \[
 \pi_{\delta\alpha}^1\prec_{X,\mathcal U_\alpha}\pi_{\delta\alpha}^0,
\]
then either
\[
 \pi_{\delta\alpha}^1\prec_{X\cap A,\mathcal U_\alpha}\pi_{\delta\alpha}^0,
\]
or
\[
 \pi_{\delta\alpha}^1\prec_{X\setminus A,\mathcal U_\alpha}\pi_{\delta\alpha}^0.
\]
This follows from the fact that either
\[
  |\pre{(\pi^1_{\delta\alpha})}{m}\cap X\cap A|\geq
  \frac{|\pre{(\pi^1_{\delta\alpha})}{m}\cap X|}{2}
\]
or
\[
  |\pre{(\pi^1_{\delta\alpha})}{m}\cap X\setminus A|\geq
  \frac{|\pre{(\pi^1_{\delta\alpha})}{m}\cap X|}{2}
\]
for \(\mathcal U_\alpha\)-many \(m\)s and that if \(s\) tends to infinity then
so does \(s/2\). The result then immediately follows because one of the two
cases has to happen for cofinally many \(\alpha<\delta\).
\end{proof}

\begin{proposition}
If \(f:\Pi_0\to_{\bar{\mathcal U}}\Pi_1\) is  \label{kill}
finite-to-one, then the set
\[
  D_f=\big\{X\in \mathbb P(\bar{\mathcal U},\Pi,\Pi_0,\Pi_1):
        (\exists Y\in\mathcal U_0)(X\cap f[(\pi_{\delta0}^0)^{-1}[Y]]=\emptyset\big\}
\]
is dense.
\end{proposition}

\begin{proof}
Let  \(X\in\mathbb P(\bar{\mathcal U},\Pi,\Pi_0,\Pi_1)\). For each
\(\alpha<\delta\) choose \(Y_\alpha\in\mathcal U_\alpha\) and
\(s_\alpha:\omega\to\omega\) tending to infinity witnessing
\(\pi^1_{\delta\alpha}\prec_{X,\mathcal U_\alpha}\pi^0_{\delta\alpha}\).
We may also assume that \(\pi^1_{\delta\alpha}(f(k))=\pi^0_{\delta\alpha}(k)\)
for all \(k\in(\pi^0_{\delta\alpha})^{-1}[Y_\alpha]\),
\(\pi_{\alpha0}(\pi^1_{\delta\alpha}(k))=\pi^1_{\delta0}(k)\) for each
\(k\in(\pi^1_{\delta\alpha})^{-1}[Y_\alpha]\) and \(\alpha<\delta\); and
\(Y_\alpha\subseteq(\pi_{\alpha0})^{-1}[Y_0]\). Since \(\mathcal U_0\)
is a P-point, there is a \(Y\in\mathcal U_0\) which is a pseudointersection
of \(\pi_{\alpha0}[Y_\alpha]\) and let \(n_\alpha<\omega\) be such that
\(\pi_{\alpha0}[Y_\alpha\setminus n_\alpha]\subseteq Y\). Also write
\(Z=(\pi^0_{\delta0})^{-1}[Y]\).
Let
\[
  X^\prime=\big((\pi^1_{\delta0})^{-1}[Y]\cap X\big)\setminus f[Z\cap X]
\]
Notice that for each \(n\in Y_\alpha\setminus n_\alpha\) we have
\[
  f^{-1}\big[\pre{(\pi_{\delta\alpha}^1)}{n}\big]\cap X
  \subseteq\pre{(\pi_{\delta\alpha}^0)}{n}\cap X
\]
(since \(f,\pi_{\delta\alpha}^0,\pi_{\delta\alpha}^1\) commute on
\(Y_\alpha\setminus n_\alpha\))
so that
\[
  \big|\pre{(\pi_{\delta\alpha}^1)}{n}\cap X\cap f[Z\cap X]\big|\leq\big|\pre{(\pi_{\delta\alpha}^0)}{n}\cap X\big|.
\]
By the choice of \(Y_\alpha\) we also have
\[
  \big|\pre{(\pi_{\delta\alpha}^1)}{n}\cap X\big|\geq s_\alpha(n)\cdot\big|\pre{(\pi_{\delta\alpha}^0)}{n}\cap X\big|
\]
Putting this together gives:
\[
  \big|\pre{(\pi_{\delta\alpha}^1)}{n}\cap X^\prime\big|\geq (s_\alpha(n)-1)\cdot\big|\pre{(\pi_{\delta\alpha}^0)}{n}\cap X^\prime\big|
\]
Since \(s_\alpha\) tends to infinity so does \(s_\alpha-1\). This
shows that \(s_\alpha-1\) and \(Y_\alpha\setminus n_\alpha\) witness the
fact that \(X^\prime\in {\mathbb P}(\bar{\mathcal U},\Pi, \Pi_0,\Pi_1)\).
\end{proof}

We now put the previous propositions together and prove:

\begin{theorem}\label{thm:unbounded}
Assume \(\Diamond\). There is a sequence of ultrafilters of length $\omega_1$, which is strictly increasing in the RK-order and has no upper bound which would be a P-point.
\end{theorem}

\begin{proof}
Let  \(\langle \Pi^\alpha:\alpha<\omega_1\rangle\), where \(\Pi^\alpha=\langle p^\alpha_{\omega_1\gamma}:\gamma<\alpha\rangle\), be a diamond sequence
guessing sequences of functions from \(\omega\) to \(\omega\), i.e. such that for
every sequence of such functions \(\Pi=\langle\pi_{\omega_1\alpha}:\alpha<\omega_1\rangle\) the set
\[
  \big\{\alpha\in Lim(\omega_1):\Pi\upharpoonright\alpha=\Pi^\alpha\big\}
\]
is stationary. We recursively construct an RK-increasing sequence
\(\langle \mathcal U_\alpha:\alpha<\omega_1\rangle\) of P-points and
witnessing maps \(\Pi_1^\alpha=\langle\pi^1_{\alpha\beta}:\beta<\alpha\rangle\) as follows.

At a successor step \(\alpha+1\), we just construct an arbitrary P-point \(\mathcal U_{\alpha+1}\) above
\(\mathcal U_\alpha\) and let \(\Pi^{\alpha+1}_1\) be the appropriate witnessing maps.

At a limit step \(\alpha\) let \(\Pi=\langle\pi^1_{\beta\gamma}:\gamma\leq\beta<\alpha\rangle\),
\(\bar{\mathcal U}=\langle\mathcal U_\beta:\beta<\alpha\rangle\), and write
\(\Pi^\alpha=\Pi^\alpha_0=\{\pi^0_{\alpha\beta}:\beta<\alpha\}\) (i.e. \(\pi^0_{\alpha\beta}=p^\alpha_{\omega_1\beta}\)).
If \(\Pi\cup\Pi^\alpha_0\) do not commute w.r.t. \(\bar{\mathcal U}\), we construct \(\mathcal U_\alpha\) to
be an arbitrary P-point above \(\bar{\mathcal U}\) and let \(\Pi^\alpha_1\) be appropriate finite-to-one
witnessing maps. Otherwise, we use \autoref{refine_maps} to construct \(\Pi^{\alpha}_1=\{\pi^1_{\alpha\beta}:\beta<\alpha\}\) satisfying

\begin{enumerate}[series=condition6]
  \item \label{cond:below} \(\Pi^\alpha_1\prec\Pi^\alpha_0\); and
\end{enumerate}

and then we recursively construct a P-filter \(\mathcal U_\alpha\) on \({\mathbb P}(\bar{\mathcal U},\Pi,\Pi^\alpha_0,\Pi^\alpha_1)\) so that

\begin{enumerate}[resume=condition6]
  \item \label{cond:disjoint} for all finite-to-one \(f:\Pi^\alpha_0\to_{\bar{\mathcal U}}\Pi^{\alpha}_1\)
  there is \(X\in\mathcal U_\alpha\) and \(Y\in\mathcal U_0\) such that
  \[
    \emptyset =
    f\big[(\pi_{\alpha0}^0)^{-1}[Y]\big]\cap X.
  \]
\end{enumerate}

To guarantee \eqref{cond:below} we just need to ensure that it hits each of the \(\omega_1\)-many
dense sets \(\{D_f:f:\Pi^\alpha_0\to_{\bar{\mathcal U}}\Pi^{\alpha}_1\}\); we also make sure that it
hits the dense sets \(\{D_A:A\subseteq\omega\}\) so that it is an ultrafilter. This can be done since
the forcing is \(\sigma\)-closed by \autoref{closed}. This finishes the recursive construction.

 Finally notice that the chain of P-points thus constructed cannot
have a P-point on top. Otherwise suppose \(\mathcal U\) is RK-above the chain as witnessed by finite-to-one
maps \(\Pi^{\omega_1}=\{\pi_{\omega_1\alpha}:\alpha<\omega_1\}\) which commute with $\bigcup_{\alpha<\omega_1} \Pi^\alpha_1$. Since the \(\Pi^\alpha\)s formed a diamond sequence, there is a limit \(\alpha<\omega_1\) such
that \(\Pi^{\omega_1}\upharpoonright\alpha=\Pi^\alpha\). Then \(\Pi^\alpha=\Pi^\alpha_0\) commutes with \(\Pi\)
so we can apply \eqref{cond:disjoint} to \(f=\pi_{\omega_1\alpha}\) and conclude that there is
\(X\in\mathcal U_\alpha\) and \(Y\in\mathcal U_0\) such that
\[
  \emptyset =
  \pi_{\omega_1\alpha}\left[(\pi^0_{\alpha0})^{-1}[Y]\right]\cap X =
  \pi_{\omega_1\alpha}\left[(p^\alpha_{\omega_10})^{-1}[Y]\right]\cap X =
  \pi_{\omega_1\alpha}\left[\pi_{\omega_10}^{-1}[Y]\right]\cap X
\]
contradicting the fact that \(\pi_{\omega_1\alpha}\) witnesses that \(\mathcal U\) is above \(\mathcal U_\alpha\).
\end{proof}
\section{Concluding Remarks} \label{sec:concluding}
It was proved in Section \ref{sec:noshort} that given a Rudin-Keisler increasing chain of rapid P-points $\seq{\UUU}{\alpha}{<}{{\omega}_{1}}$ together with a commuting sequence of finite-to-one witnessing maps $\langle {\pi}_{\beta, \alpha}: \alpha \leq \beta < {\omega}_{1} \rangle$, it is possible to find a sequence of finite-to-one maps $\seq{g}{\alpha}{<}{{\omega}_{1}}$ together with a rapid P-point $\VVV$ such that ${g}_{\alpha}$ is a witness to ${\UUU}_{\alpha} \: {\leq}_{\mathrm{RK}} \: \VVV$.
However the argument in Section \ref{sec:noshort} does not guarantee that any of the ${g}_{\alpha}$ will be nondecreasing even when it is given that each of the maps ${\pi}_{\beta\alpha}$ is nondecreasing.
In other words, the rapid P-point $\VVV$ may not be an ${\leq}^{+}_{\mathrm{RB}}$ upper bound of the sequence $\seq{\UUU}{\alpha}{<}{{\omega}_{1}}$ even if that sequence itself is assumed to be ${\leq}^{+}_{\mathrm{RB}}$-increasing.

It appears that one must fall back on the construction given in \cite{Raghavan:2016} if one wants a chain of P-points of length ${\cc}^{+}$ which is increasing in the ${\leq}^{+}_{\mathrm{RB}}$ ordering.
Nevertheless the ideas from Section \ref{sec:noshort} can be combined with the work in \cite{Raghavan:2016} to show that $\CH$ implies that the rapid P-points are ${\cc}^{+}$-closed with respect to ${\leq}^{+}_{\mathrm{RB}}$.
More precisely, the following theorem will appear in a forthcoming paper of Kuzeljevi{\'c}, Raghavan, and Verner~\cite{krv}: Assume the Continuum Hypothesis.
Suppose $\delta < {\cc}^{+}$.
If $\seq{\UUU}{\alpha}{<}{\delta}$ is any sequence of rapid P-points which is increasing with respect to ${\leq}^{+}_{\mathrm{RB}}$, then there exists a rapid P-point $\VVV$ such that $\forall \alpha < \delta\left[{\UUU}_{\alpha} \: {\leq}^{+}_{\mathrm{RB}} \: \VVV\right]$.
Therefore every strictly increasing sequence of rapid P-points of length $< {\cc}^{+}$ can be extended to one of length ${\cc}^{+}$ with respect to the ${\leq}^{+}_{\mathrm{RB}}$ ordering.
\bibliographystyle{amsplain}
\bibliography{main}
\end{document}